\documentclass[12pt]{amsart}
\usepackage{amssymb,amsmath,amsthm,cases}
\numberwithin{equation}{section}
\begin{document}

\theoremstyle{plain}
\newtheorem{theorem}{Theorem}[section]
\newtheorem{lemma}[theorem]{Lemma}
\newtheorem{proposition}[theorem]{Proposition}
\newtheorem{corollary}[theorem]{Corollary}
\newtheorem{conjecture}[theorem]{Conjecture}

\def\mod#1{{\ifmmode\text{\rm\ (mod~$#1$)}
\else\discretionary{}{}{\hbox{ }}\rm(mod~$#1$)\fi}}

\theoremstyle{definition}
\newtheorem*{definition}{Definition}

\theoremstyle{remark}
\newtheorem{remark}{Remark}[section]
\newtheorem{example}{Example}[section]
\newtheorem*{remarks}{Remarks}
\newcommand{\ndiv}{\hspace{-4pt}\not|\hspace{2pt}}
\newcommand{\cc}{{\mathbb C}}
\newcommand{\qq}{{\mathbb Q}}
\newcommand{\rr}{{\mathbb R}}
\newcommand{\nn}{{\mathbb N}}
\newcommand{\zz}{{\mathbb Z}}
\newcommand{\pp}{{\mathbb P}}
\newcommand{\al}{\alpha}
\newcommand{\be}{\beta}
\newcommand{\ga}{\gamma}
\newcommand{\ze}{\zeta}
\newcommand{\ka}{\kappa}
\newcommand{\om}{\omega}
\newcommand{\mz}{{\mathcal Z}}
\newcommand{\mi}{{\mathcal I}}
\newcommand{\ep}{\epsilon}
\newcommand{\la}{\lambda}
\newcommand{\de}{\delta}
\newcommand{\tha}{\theta}
\newcommand{\De}{\Delta}
\newcommand{\Ga}{\Gamma}
\newcommand{\si}{\sigma}
\newcommand{\Exp}{{\rm Exp}}
\newcommand{\legen}[2]{\genfrac{(}{)}{}{}{#1}{#2}}
\def\End{{\rm End}}
\title{Equal sums of two cubes of quadratic forms}
\dedicatory{This paper is dedicated to my friend and colleague Bruce Berndt on his 80th birthday.}
\author{Bruce Reznick}
\address{Department of Mathematics, University of 
Illinois at Urbana-Champaign, Urbana, IL 61801} 
\email{reznick@math.uiuc.edu}
\date{\today} 
\begin{abstract}
We give a complete description of all solutions to the equation $f_1^3 + f_2^3 = f_3^3 + f_4^3$
for quadratic forms $f_j \in \mathbb C[x,y]$ and show how Ramanujan's example can be 
extended to three equal sums of pairs of cubes. We also give a complete census in counting
the number of ways a sextic $p \in \mathbb C[x,y]$ can be written as a sum of two cubes. The
extreme example is $p(x,y) = xy(x^4-y^4)$, which has six such representations.
\end{abstract}

\thanks{The author was supported by Simons Collaboration Grant 280987.}

\maketitle

\section{Introduction} 
In 1913, Ramanujan \cite{Ram}, \cite[p.326]{Ram2} (see \cite[p.56]{B},
\cite[p.201]{HW}) posed to the {\it Journal of
the Indian Mathematical Society} the following question: ``Shew that
\begin{equation}\label{E:Ram}
\begin{gathered}
(6x^2-4xy+4y^2)^3 = \\
(3x^2+5xy-5y^2)^3 +  (4x^2 -4xy+6y^2)^3 +
(5x^2-5xy-3y^2)^3,
\end{gathered}
\end{equation}
and find other quadratic expressions satisfying similar relations.'' Write  \eqref{E:Ram}
as $R_1^3(x,y) = R_2^3(x,y)  + R_3^3(x,y)  + R_4^3(x,y) $ for short. 

In 1914, Narayanan \cite{Nar} replaced the integers in \eqref{E:Ram} with the variables 
$\ell,m,n,p$ and solved the resulting equations; namely, 
$m^3+n^3=p^3 -\ell^3 = m p^2 + n\ell^2$, over $\mathbb R$.
 \begin{equation}\label{E:Nar}
\begin{gathered}
(\ell x^2 - n x y + ny^2)^3 = \\
(px^2 + mxy - my^2)^3 +   (nx^2 - nxy + \ell y^2)^3 + (mx^2 - mxy -p y^2)^3; \\
\ell = \la(\la^3 + 1),\quad m = 2\la^3 - 1,\quad n = \la(\la^3
-2),\quad  p = \la^3 + 1.
\end{gathered}
\end{equation}
Write   \eqref{E:Nar}
as $N_{1,\la}^3(x,y) = N_{2,\la}^3(x,y)  + N_{3,\la}^3(x,y)  +N_{4,\la}^3(x,y) $,
and note $N_{j,2} = 3 R_j$.

Equation \eqref{E:Ram} can be rewritten as two equal sums of two cubes in three different
ways, and in two of the three ways, there is a third equal sum of two cubes. First,
\begin{equation}\label{Flip1}
\begin{gathered}
(4x^2 -4xy+6y^2)^3 + (5x^2-5xy-3y^2)^3 \\
=(6x^2-4xy+4y^2)^3  -(3x^2+5xy-5y^2)^3 \\ 
 =(6x^2 - 8xy + 6y^2)^3 - (3x^2-11xy+3y^2)^3 \\
 =63(x^2 + xy + y^2)(3x^2  - 3xy + y^2)(x^2 - 3xy +3y^2).
\end{gathered}
\end{equation}

We also have
\begin{equation}\label{Flip2}
\begin{gathered}
(6x^2-4xy+4y^2)^3 - (5x^2-5xy-3y^2)^3  \\
= (4x^2 -4xy+6y^2)^3 +(3x^2+5xy-5y^2)^3\\ 
 = \left( \tfrac {94}{21}x^2 - \tfrac 8{21}xy+ \tfrac {94}{21}y^2\right)^3
+ \left( \tfrac {23}{21}x^2 - \tfrac{199}{21}xy+ \tfrac
{23}{21}y^2\right)^3 \\
= (13x^2 - 23 xy + 13y^2)(7x^2 + xy + y^2)(x^2+xy + 7y^2),
\end{gathered}
\end{equation}
and
\begin{equation}\label{Flip3}
\begin{gathered}
(6x^2-4xy+4y^2)^3 - (4x^2 -4xy+6y^2)^3\\
= (3x^2+5xy-5y^2)^3 + (5x^2-5xy-3y^2)^3 \\ 
= 8(x-y)(x+y)(x^2-xy + y^2)(19x^2 -11xy + 19y^2).
\end{gathered}
\end{equation}
It can be shown that there is no third representation in \eqref{Flip3}.
Furthermore, \eqref{Flip2} follows from \eqref{Flip1} (with the rows permuted)
upon making the unimodular linear change of variables:
$\left(x,y\right) \to ( \frac{5x-2y}{\sqrt{21}},
\frac{3x+3y}{\sqrt{21}})$.

Comparable versions of these properties apply to the Narayanan
formulas (see \eqref{E:Naren}).  More to the point, up to transposition of terms,
changes of variable and taking $\la \in 
\mathbb C$, we shall show that \eqref{E:Nar} 
{\it completely} describes the solution in  binary quadratic forms $f_j = f_j(x,y) \in \mathbb C[x,y]$
to
\begin{equation}\label{E:funda}
p = f_1^3 + f_2^3 = f_3^3 + f_4^3.
\end{equation}
Our analysis comes from looking at the equation in quadratic forms over $\mathbb C$
and studying the properties of the common sum $p$. 

We begin with some notations, following those in \cite{Re3}. For $m \ge 3$,
let $\ze_m=  e^{\frac{2\pi i}m}$ and $\om = \ze_3$.
Two forms in $\cc[x,y]$ are  {\it distinct} 
if they are not proportional. The identity \eqref{E:funda} is {\it honest} if 
the $f_j$'s are pairwise distinct. A {\it flip} of \eqref{E:funda} is either of the two equivalent identities
\begin{equation}\label{E:flip}
p_1 = f_1^3 - f_3^3 = -f_2^3 + f_4^3, \qquad p_2 = f_1^3 - f_4^3 = -f_2^3 + f_3^3.
\end{equation}
There seems to be no obvious way of deriving $p_1$ or $p_2$ from $p$ in \eqref{E:flip}. If
\eqref{E:funda} holds, we say that the family $\mathcal F = \{\{f_1,f_2\}, \{f_3,f_4\}\}$
{\it represents} $p$,  with the understanding that
two families $\mathcal F$ and $\mathcal G$ are identified if 
$\{\{f_1^3,f_2^3\}, \{f_3^3,f_4^3\}\}  = \{\{g_1^3,g_2^3\}, \{g_3^3,g_4^3\}\}$; we do not care
about the order of the summands, or powers of $\om$ multiplying the quadratics.  
For a sextic form $p \in \mathbb C[x,y]$,
we define $N(p)$ to be the number of pairwise-nonsimilar families $\mathcal F$ representing
$p$.

 If $M(x,y) = (\al x + \be y, \ga x + \de y)$ is an invertible linear change of variables 
(or {\it linear change} for short), and $f \in \cc[x,y]$ is a form, define $f \circ M$ by 
 $(f \circ M)(x,y)=f(\al x + \be y, \ga x + \de y)$.
A {\it scaling} is a linear change in which $\be = \ga = 0$. If $\deg f = d$, and $\de = \al$ in
a scaling $M$, then $f\circ M = \al^d f$, If $M$ is a linear change,
and $g = f \circ M$, then $f$ and $g$ are {\it similar}, and if $\mathcal G= \mathcal F\circ M$,
the $\mathcal F$ and $\mathcal G$ will also be called similar. 

More generally, suppose the equation 
\begin{equation}\label{E:onerep}
p = f_1^3 + f_2^3
\end{equation}
holds. If $M$ is a linear change, then \eqref{E:onerep} implies that 
$p\circ M = (f_1\circ M)^3 + (f_2\circ M)^3$. 
It may happen that $p = p \circ M$, but that $\{(f_1\circ M)^3, (f_2\circ M)^3\} \neq
\{f_1^3,f_2^3\}$: this seems to be the inherent mechanism behind
multiple representations.

The following underlying identity is central to our analysis. For $\al \in \cc$, 
\begin{equation}\label{E:threefold}
(\al x^2 - xy + \al y^2)^3 +\al( - x^2 +\al xy - y^2)^3 =  (\al^2-1)(\al x^3 + y^3)(x^3 + \al y^3).
\end{equation}
(This can easily be verified by setting $v = x^2 + y^2$ and $w = xy$ and noting that $v^3-3vw^2 
= x^6 + y^6$.)  Observe that the sum is a quadratic in $\{x^3,y^3\}$, and so if $(x,y) \mapsto
 (\om x, \om^2 y)$, then the sum is unchanged, although the summands are changed. 
 Writing $\al = \la^3$, we can bring in the outside coefficient and obtain
 \begin{equation}\label{E:threefold2}
\begin{gathered}
(\la^3 x^2 - xy + \la^3 y^2)^3 +( - \la x^2 +\la^4 xy - \la y^2)^3 \\ 
= (\la^3 \om^2 x^2 - xy + \la^3 \om y^2)^3 +( - \la \om^2 x^2 +\la^4 xy - \la \om y^2)^3\\
= (\la^3 \om x^2 - xy + \la^3 \om^2y^2)^3 +( - \la \om x^2 +\la^4 xy - \la \om^2 y^2)^3
 \\ = p_{1,\la}(x,y):=  (\la^6-1)(\la^3 x^3 + y^3)(x^3 + \la^3y^3).
\end{gathered}
\end{equation}
Write the summands in \eqref{E:threefold2} as:
\begin{equation}\label{E:naming}
\begin{gathered}
F_{1,\la}(x,y) = \la^3 x^2 - xy + \la^3 y^2, \quad F_{2,\la}(x,y) =  - \la x^2 +\la^4 xy - \la y^2,\\
F_{3,\la}(x,y) = F_{1,\la}(\om x, \om^2 y), \quad F_{4,\la}(x,y) = F_{2,\la}(\om x, \om^2 y),\\
F_{5,\la}(x,y) = F_{1,\la}(\om^2 x, \om y), \quad F_{6,\la}(x,y) = F_{2,\la}(\om^2 x, \om y).
\end{gathered}
\end{equation}
 If $\la = 0$ or $\la^6 = 1$, then the identities of \eqref{E:threefold2} 
 are not honest, so we shall assume that $\la(\la^6-1) \ne 0$.
 Let $\mathcal F_{1,\la} =\{\{F_{1,\la},F_{2,\la}\}, \{F_{3,\la},F_{4,\la}\}\}$,
 $\mathcal F_{2,\la} = \{\{F_{3,\la},F_{4,\la}\}, \{F_{5,\la},F_{6,\la}\}\}$ and
 $\mathcal F_{3,\la} = \{\{F_{5,\la},F_{6\la}\}, \{F_{1,\la},F_{2,\al}\}\}$.
 Observe that under the scaling  $(x,y) \mapsto
 (\om x, \om^2 y)$, $\mathcal F_{1,\la} \mapsto \mathcal F_{2,\la} \mapsto 
 \mathcal F_{3,\la} \mapsto \mathcal F_{1,\la}$. Thus the three sets of equations 
 $ F_{1,\la}^3 + F_{2,\la}^3 =   F_{3,\la}^3 + F_{4,\la}^3$, 
 $F_{1,\la}^3 + F_{2,\la}^3 =   F_{5,\la}^3 + F_{6,\la}^3$, and
 $F_{3,\la}^3 + F_{4,\la}^3=F_{5,\la}^3 + F_{6,\la}^3$
 are similar to each other. The ``cleanest" versions of the flips come from $\mathcal F_{2,\la}$:
   \begin{equation}\label{E:flippo1}
\begin{gathered}
F_{4,\la}^3(x,y)  - F_{5,\la}^3(x,y) =  -F_{3,\la}^3(x,y) + F_{6,\la}^3(x,y) =p_{2,\la}(x,y) :=\\
((1+\la^6)x^3 +3\la^3 x^2 y - \la^3 y^3)(-\la^3 x^3 + 3\la^3 x y^2 +(1 + \la^6)y^3);
\end{gathered}
\end{equation}
 \begin{equation}\label{E:flippo2}
\begin{gathered}
 F_{4,\la}^3(x,y) - F_{6,\la}^3(x,y) = F_{5,\la}^3(x,y) - F_{3,\la}^3(x,y) = p_{3,\la}(x,y) :=\\
3\sqrt{-3}\ x y (x-y)(x+y)(\la^3 x + y)(x + \la^3 y).
\end{gathered}
\end{equation}

We now present some symmetries of  \eqref{E:threefold2}. Since
$F_{j,-\la}(x,y) = -F_{j,\la}(x,-y)$,  $\mathcal F_{j,-\la}$ is similar to $\mathcal F_{j,\la}$.
Further, $F_{1,\la^{-1}} = -\la^{-4}F_{2,\la}$ and $F_{2,\la^{-1}} = -\la^{-4}F_{1,\la}$, etc., so
 $\mathcal F_{j,\la^{-1}}$ is similar to $\mathcal F_{j,\la}$.
Under the unimodular linear change 
\[(x,y) \mapsto \tfrac 1{\sqrt{1-\la^6}}(\la^3 x + y, -(x + \la^3 y)),
\]
the system of identities
\[
\begin{gathered}
F_{1,\la}^3(x,y) +F_{2,\la}^3(x,y) = F_{3,\la}^3(x,y) + F_{4,\la}^3(x,y)
  =F_{5,\la}^3(x,y) + F_{6,\la}^3(x,y)
\end{gathered}
\]
becomes
\[
\begin{gathered}
F_{7,\la}^3(x,y) + F_{8,\la}^3(x,y) = -F_{3,\la}^3(x,y) + F_{6,\la}^3(x,y) 
=  -F_{5,\la}^3(x,y) + F_{4,\la}^3(x,y); \\
F_7(x,y) = \tfrac 1{1-\la^6}\left( (2\la^3 + \la^9)x^2 + (1 + 5\la^6) x y + (2\la^3 + \la^9) y^2 \right),\\
F_8(x,y) = -\tfrac {\la}{1-\la^6}\left( (1 + 2\la^6)x^2 + (5\la^3 + \la^9) x y + (1 + 2\la^6)y^2 \right). 
\end{gathered}
\]
Of course, $p_1(x,y) \mapsto p_2(x,y)$ under this linear change. This means that each
$\mathcal F_{j,\la}$ is similar to one of its flips.

If we make the linear change $(x,y) \mapsto (x + \om^2 y, x + \om y)$ into 
\eqref{E:threefold2}, we obtain an enhanced version of \eqref{E:Nar}, with a third sum:
\begin{equation}\label{E:Naren}
\begin{gathered}
N_{4,\la}^3(x,y) + N_{3,\la}^3(x,y) = -N_{2,\la}^3(x,y) + N_{1,\la}^3(x,y) \\
= (-p x^2 + (m+2p) x y - p y^2)^3 + (\ell x^2 + (n - 2\ell)x y + \ell y^2)^3.
\end{gathered}
\end{equation}
Upon continuing with the linear change which takes \eqref{E:threefold2} into \eqref{E:flippo1}, we
get a flipped version of \eqref{E:Nar} and another third equal sum, but  with denominators.
A slightly different linear change gives a simple version in
$\mathbb Q(\la)[x,y]$: under $(x,y) \mapsto (x - \sqrt{-3}\ y, x + \sqrt{-3}\ y)$, 
and multiplication by $-1$, \eqref{E:threefold2} becomes
\[
\begin{gathered}
((1 -2 \la^3) x^2 + 3 (1 + 2 \la^3)y^2)^3 +  (\la(2-\la^3) x^2 - 3\la (2 + \la^3) y^2)^3 \\ 
= ((1 + \la^3) x^2 + 6 \la^3 x y + 3 (1-\la^3) y^2)^3 + ( -\la(1 + \la^3) x^2 - 6\la x y + 3 \la(1 -\la^3) y^2)^3 \\
= ((1 + \la^3) x^2 - 6 \la^3 x y + 3 (1-\la^3) y^2)^3 + ( -\la(1 + \la^3) x^2 + 6\la x y + 3 \la(1 -\la^3) y^2)^3. 
\end{gathered}
\]

It is also worth noting that under the  linear change 
$(x,y) \mapsto (x+\tau y, -i(\tau x - y))$, $\tau = \sqrt{1-\la^6} - i\la^3$, (which is invertible
provided $\la^6 \neq 1$), \eqref{E:flippo2} becomes an equation of the shape
$(ax^2 +b xy + a y^2)^3 + (ax^2 -b xy + a y^2)^3 = (rx^2 + sy^2)^3 + (sx^2+ry^2)^3$, and 
$p_{3,\la}$ becomes a multiple of $x^6 +(4\la^6 -1)x^4y^2 + (4\la^6 -1)x^2y^4 + y^6$. This
phenomenon is explored in Theorem \ref{T:tamerep}.

This paper has two parts. The main result of the first part is the following theorem. 
 \begin{theorem}\label{T:cubic}
Every honest identity \eqref{E:funda} for binary sextics is similar to 
some $\mathcal F_{2,\la}$ with $\la(\la^6-1) \neq 0$, up to a possible flip.
\end{theorem}

There is a crucial intermediate step in the proof of Theorem \ref{T:cubic}. Any four
binary quadratic forms are linearly dependent, and a given dependence is not affected by a
linear change. We shall say that an honest 
 \eqref{E:funda} is an identity of {\it Type$(T)$} if, perhaps
after a flip, the following two equations hold:
\begin{equation}\label{E:type}
f_1^3 + f_2^3 = f_3^3 + f_4^3, \qquad f_1 + f_2 = T(f_3 + f_4).
\end{equation}
We show (see Lemma \ref{L:notype}) that $T(T^3-1) \neq 0$ in an honest family of
Type$(T)$. Of course, 
the same equation is both Type$(T)$ and Type$(T^{-1})$, and factors of $\om^k$ do not matter. 

The following identities show that \eqref{E:Nar} and \eqref{E:threefold} are both
of Type$(\la^2)$:
\[
\begin{gathered}
N_{2,\la}^3 +N_{4,\la}^3 = N_{1,\la}^3-N_{3,\la}^3, \quad 
N_{2,\la} + N_{4,\la} = \la^2(N_{1,\la} - N_{3,\la}); \\
F_{5,\la}^3 -F_{3,\la}^3 = F_{4,\la}^3 - F_{6,\la}^3,  \quad
F_{5,\la} -F_{3,\la} = \la^2(F_{4,\la} -F_{6,\la}).
\end{gathered}
\]

We prove Theorem \ref{T:cubic} in two stages. After a few technical lemmas,  we show that after 
a linear change, for any honest solution \eqref{E:funda},  
$f_1$ and $f_2$ are both even and that $f_3$
and $f_4$ are not (see Corollary \ref{C:evennot}). We then determine all honest \eqref{E:funda}
in which $f_3,f_4$ are not both even, but $f_3^3 + f_4^3$ is 
(see Theorems \ref{T:tamerep}, \ref{wildcase}) and show that they must be of Type$(T)$ for
some $T$. (Geometrically, this says that any quadratic curve which lies on the surface 
$z_1^3 + z_2^3 + z_3^3 + z_4^3 = 0$ must in fact lie on the intersection of the surface with
a hyperplane $z_i + z_j + T(z_k + z_{\ell}) = 0$) for some permutation of the indices.) 
We finally show that any two honest solutions of \eqref{E:funda}  of Type$(T)$ are 
similar (Theorem \ref{miracle}), and are similar to \eqref{E:threefold2} (or \eqref{E:Nar}) with
$T = \la^2$. We also explore solutions to  \eqref{E:funda} with $f_j \in \mathbb Q[x,y]$. If such
an equation has type $T = \la^2$, then it is clear that $T \in \mathbb Q$; \eqref{E:Naren}
shows that such a solution occurs when $\la \in \mathbb Q$. In Theorem \ref{T:lastminute},
we show that no rational solution can occur when $T < 0$ or $T = 2$. We suspect that
$\sqrt{T} \in \mathbb Q$ is also necessary, but hope to be proved wrong. 

In the second part of the paper, 
we give a complete description of $N(p)$,  the number of different
ways that a binary sextic form is a sum of two cubes. A key result
(see Theorem \ref{T:B})  is that a form $p$ (of degree $3k$) is a sum of two cubes
if and only if $p=h_1h_2h_3$ where the $h_j$'s are distinct, but linearly dependent. There are
two important families of sextics: for $t \in \mathbb C$, let 
\begin{equation}\label{E:basics}
\begin{gathered}
A_t(x,y) := x^6 + tx^4y^2 + tx^2y^4 + y^6, \qquad
B_t(x,y): = x^6 + tx^3y^3 + y^6.
\end{gathered}
\end{equation}
Observe that $p_{1,\la} = \la^3(\la^6-1)B_{\la^3 + \la^{-3}}$, and as we have seen, $p_{3,\la}$
is similar to $A_{4\la^6-1}$. Every $A_t$ and $B_t$ is thus similar to $p_{1,\la}$ or $p_{3,\la}$
for $\la$ with $\la(1-\la^6) \neq 0$ except for $A_{-1},A_{3},B_{\pm 2}$. 

We give a census of $N(p)$ for binary sextics: (i) a binary sextic $p$ is a sum of two cubes (that is,
$N(p) \ge 1$) if and only if $p = \ell^3q$, where $\ell$ is linear and $q$ is a square-free cubic
or $p$ is similar to $q(x^2,y^2)$, where $q$ is a square-free cubic (see Theorem \ref{T:1});
(ii) a binary sextic $p$ has $N(p) =2$ if and only if $p$ is similar to $A_t$ for $t \in \cc$, except
that $N(A_3) = 0$,  $N(A_{-1})= 1$, $N(A_0) = N(A_{15}) = 4$ and $N(A_{-5}) = 6$ (see
Theorem \ref{T:2}); (iii) a binary sextic $p$ has $N(p) = 3$ if and only if $p$ is similar to $B_t$ 
for $t \in \cc$, except that $N(B_{\pm 2}) = 0$, $N(B_0)= 4$ and $N(B_{\pm 5\sqrt{-2}}) = 6$,
(see Theorem \ref{T:3}); (iv) up to similarity, there are two sextics with $N(p) > 3$:
\begin{equation}\label{E:biggies}
\begin{gathered}
Q_1(x,y) = x^6 + y^6\quad \text{or}\quad Q_2(x,y) = xy(x^4-y^4). 
\end{gathered}
\end{equation}
To specific, $Q_1$ is similar to $A_0,A_{15}, B_0$ and $N(Q_1) = 4$ and 
$Q_2$ is similar to $A_{-5}$ and $B_{\pm 5\sqrt{-2}}$ and $N(Q_2) = 6$ (see Theorem \ref{T:46}). 
Section six gives some extra attention to the representations of $Q_1, Q_2$ and their
similarities.

In the final section, we give some different directions that this study might go. We show that the classical
Euler-Binet parameterization to $a^3 + b^3 = c^3 + d^3$ over $\mathbb Q$ 
is also valid over $\mathbb C(x_1,\dots,x_n)$ (see Theorem \ref{T:EB} and Corollary
\ref{C:EB}). We apply the usual ``point addition" of points on the curve 
$x_1^3 + x_2^3 = x_3^3 + x_4^3 = A$
 to show that  $(F_{1,\la},F_{2,\la}) ``+"
(F_{3,\la},F_{4,\la}) = (F_{5,\la},F_{6,\la})$ (see Theorem \ref{T:add}); the denominators disappear.
We show, separately, that a flip of the Euler-Binet parameterization can be added to find
a third representation as a sum of cubes of polynomials (see \eqref{E:curveadd}.)
Finally, we present a few results from the huge literature. 
 We have not found a systematic analysis of \eqref{E:funda} over $\mathbb C[x,y]$, nor
\eqref{E:threefold} nor  any three-fold identities, but mention some of the other quadratic
parameterizations.

This project began 20 years ago when Bruce Berndt gave a seminar at Illinois about
\eqref{E:Ram} and \eqref{E:Nar}. The author foolishly believed that 
an algebraic approach would easily lead to all solutions, and posted a proof-free online set of notes
\cite{Re1} in 2000. Eventually, it has produced this article and an earlier companion
paper studying higher powers, \cite{Re3}.
He wishes to thank his present and former colleagues Michael Bennett, Bruce Berndt, 
 Nigel Boston, Dan Grayson and Jeremy Rouse for helpful conversations, and
Andrew Bremner, Noam Elkies and Michael Hirschhorn for encouraging and
useful emails over the years.

\section{Preliminary lemmas}

We begin with several old simple lemmas, giving proofs for completeness.
The first is a special case of, for example,  \cite[Thm.1.1]{Re3}.
\begin{lemma}\label{L:no2sums}
If $\{\al_k x + \be_k y)\}, 1 \le k \le r \le 4$ are pairwise distinct linear forms,
then $\{(\al_k x + \be_k y)^3\}$ is linearly independent. In particular, if 
\eqref{E:funda} holds and $\{h_j\}$ is honest, then it cannot be the case that each $h_j$ is even.
\end{lemma}
\begin{proof}
If $r < 4$, add more distinct linear forms to assume that $r=4$. The matrix of
$\{(\al_j x + \be_j y)^3)\}$ with respect to the basis $\{\binom 3i x^{3-i}y^i\}$ is 
$[\al_j^{3-i}\be_j^i]$, which is Vandermonde, with determinant  
$\prod_{1 \le j < k \le 4} (\al_j\be_k - \al_k\be_j)$.
This determinant is non-zero because each pair of linear forms is distinct. 

Suppose $p$ is a cubic form and
\begin{equation}\label{E:4cubics}
p(x,y) = (\al_1 x + \be_1 y)^3 + (\al_2 x + \be_2 y)^3 = (\al_3 x + \be_3 y)^3 + (\al_4 x + \be_4 y)^3.
\end{equation}
Then $0 = p-p$ gives a formal linear dependence of four cubics, which must result from pairwise 
cancellation; that is, the original representations were the same.

Finally, by comparing coefficients, the equation
\[
(\al_1 x^2 + \be_1 y)^3 + (\al_2 x^2 + \be_2 y^2)^3 = (\al_3 x^2 + \be_3 y^2)^3 +
 (\al_4 x^2 + \be_4 y^2)^3.
\]
implies \eqref{E:4cubics}, and so cannot happen in an honest family.
 \end{proof}
 \begin{lemma}\label{L:linin}
Suppose $g_1, g_2 \in \cc[x_1,\dots,x_n]$ are distinct forms. Then for $d \ge 2$, the set
$\{g_1^{d-k}g_2^k:0 \le k \le d\}$ is linearly independent.
\end{lemma}
\begin{proof}
Suppose $\sum_{k=0}^d \la_kg_1^{d-k}g_2^k = 0$ for a non-zero choice of $\{\la_k\}$. Then
 \[
 \sum_{k=0}^d \la_kx^{d-k}y^k = \prod_{j=1}^d(\al_j x + \be_j y) \implies
\prod_{j=1}^d(\al_j g_1 + \be_j g_2) = 0;
\]
thus $\al_j g_1 + \be_j g_2 = 0$ for some $j$, violating the distinctness hypothesis.
\end{proof}
 
 We need an old fact about simultaneous
 diagonalization; there doesn't seem to be a standard easy-to-find modern proof, a different proof
 is shown in \cite[Thm.3.2]{Re3}.
 
\begin{theorem}\label{T:diag}
If $f_1(x,y)$ and $f_2(x,y)$ are relatively prime  quadratic forms,
then there is a linear change $M$ so that $f_1\circ M$ and $f_2\circ M$ are both even. 
\end{theorem}
\begin{proof}
We may assume $rank(f_1) \ge
rank(f_2)\ge 1$, and after a preliminary linear change, 
take $f_1(x,y) = x^2$ or $x^2 + y^2$. In the first case,
$rank(f_2) = 1$, so $f_2 = \ell^2$ for a linear $\ell$ which can become $y$ after a linear 
change, so $(f_1,f_2) \to (x^2,y^2)$. Otherwise,  we have $f_1(x,y) = x^2 + y^2$ and 
$f_2(x,y) = ax^2 + 2bxy + cy^2$. Since $f_1$ and $f_2$ are relatively prime, $x\pm iy$ is 
not a factor of $f_2$ and so $a \pm 2bi - c \neq0$.

The quadratic $\la f_1 + f_2$ has discriminant 
\[
\begin{gathered}
\Delta(\la)= 4(\la + a)(\la + c) - (2b)^2  = 4(\la^2 +(a+c)\la + (ac-b^2); \\
Disc(\Delta(\la)) = (a+c)^2 - 4(ac-b^2) = (a+2bi -c)(a-2bi-c) \neq 0.
\end{gathered}
\]
Thus there exist $\la_1 \neq \la_2$ so that each quadratic $\la_j f_1 + f_2$ is perfect square; that 
is, $\la_j f_1 + f_2 = \ell_j^2$. This implies that both $f_1$ and $f_2$ are linear 
combinations of $\ell_1^2, \ell_2^2$. A linear change taking $(\ell_1,\ell_2) \mapsto (x,y)$
completes the diagonalization.
\end{proof}

In order to apply Theorem \ref{T:diag}, we need a small technical lemma.
\begin{lemma}\label{L:relprime}
Suppose $p = f_1^3 + f_2^3 = f_3^3 + f_4^3$ for quadratic $f_1,f_2$ and $f_1$ and $f_2$
have a non-trivial common factor. 
Then $\{f_1^3,f_2^3\} = \{f_3^3,f_4^3\}$. Thus in any honest instance of
 \eqref{E:funda}, the $f_j$'s are pairwise relatively prime.
 \end{lemma}
 \begin{proof}
 Suppose $\ell$ is a linear form and $f_1 = \ell \ell_1$ and $f_2 = \ell \ell_2$. Then
 \[
 \ell^3\ |\ f_3^3 + f_4^3 = (f_3+f_4)(f_3 + \om f_4) (f_3 + \om^2 f_4).
 \]
 Since the three factors on the right are quadratic, $\ell$ must divide at least two of them; it
 follows that $\ell$ divides both $f_3$ and $f_4$. By writing
 $f_3 = \ell \ell_3$ and $f_4 = \ell \ell_4$, we see that $\ell_1^3 + \ell_2^3 = \ell_3^3 + \ell_4^3$, 
 and
 since the original equation was honest, the $\ell_j$'s are pairwise distinct. This is impossible
 by Lemma  \ref{L:no2sums}.
  \end{proof}
  
  Putting the results of this section together, we have the following corollary.
  \begin{corollary}\label{C:evennot}
  If an honest \eqref{E:funda} holds, then after a linear change, $f_1$ and
  $f_2$ are even, (and hence so is $p$), but $f_3$ and $f_4$ are not both even; thus
  \begin{equation}\label{E:noteven}
  (ax^2 + bxy + cy^2)^3 + (dx^2 + exy + fy^2)^3
  \end{equation}
  is even, where $(b,e) \neq (0,0)$. 
  \end{corollary}

  \section{Even sums of the cubes of non-even quadratic forms}
  
  Our goal in this section is to show that every quadratic solution
  to \eqref{E:funda} is a family of Type$(T)$ for some $T$.
  
  How can it happen that $f_3^3+ f_4^3$ is even when at least one of $\{f_3,f_4\}$ is not even? 
An obvious case is 
\begin{equation}\label{E:tame}
f_3(x,y) = a x^2 + b x y + c y^2, \qquad f_4(x,y) =  a x^2 - b x y + c y^2,
\end{equation}
which, as in \cite{Re3}, we call the {\it tame} case; otherwise we are in the
{\it wild} case. 
If $a = 0$, then it follows from \eqref{E:tame} that $y$ divides
$f_3$ and $f_4$, and by 
Lemma \ref{L:relprime}, this cannot happen, so $a \neq 0$. Similarly, $c \neq 0$. 
Thus, we may scale $x$ and $y$ and assume that $f_3,f_4$ are $x^2 \pm \ga xy + y^2$ for
some $\ga \neq 0$.
\begin{theorem}\label{T:tamerep}
The tame case occurs in a family of Type$((1 + \frac 34 \ga^2)^{1/3})$.
\end{theorem}
\begin{proof}
Observe that 
\begin{equation}\label{E:tameo}
\begin{gathered}
(x^2 + \ga x y + y^2)^3 + (x^2 - \ga x y + y^2)^3  = \\
2(x^6 + 3(1+\ga^2)x^4y^2 + 3(1+\ga^2)x^2y^4 + y^6)  = 2A_{3(1+\ga^2)}(x,y).
\end{gathered}
\end{equation}
Let $\{f_3(x,y),f_4(x,y)\} = \{x^2 \pm \ga x y + y^2\}$. Honesty requires $\ga \neq 0$. By
hypothesis, $ 2A_{3(1+\ga^2)}$ is a sum of cubes of two even quadratics in a unique way
by Lemma \ref{L:no2sums}.

Note that \eqref{E:tameo} implies that
\begin{equation}\label{E:step1}
\begin{gathered}
2(x^6 + 3(1+\ga^2)x^4y^2 + 3(1+\ga^2)x^2y^4 + y^6)  = (r_\ga x^2 + s_\ga y^2)^3 + (s_\ga x^2 + r_\ga y^2)^3
\\ \iff r_\ga^3 + s_\ga^3 = 2, \quad 3r_\ga^2s_\ga+ 3r_\ga s_\ga^2 = 3r_\ga s_\ga(r_\ga+s_\ga)= 
6(1+\ga^2)\\ \implies (r_\ga+ s_\ga)^3 = 8 + 6 \ga^2.
\end{gathered}
\end{equation}
Observe that if $\ga^2 =  -\frac 43$, then $0 = (r_\ga + s_\ga)^3$, so $s_\ga = -r_\ga$ and
 $ r_\ga^3 + s_\ga^3 = 0$, so we take $\ga^2 \neq  -\frac 43$.
Up to $(r_\ga,s_\ga) \mapsto \om^k(r_\ga,s_\ga)$ and a choice of cube root,  
\[
 r_\ga + s_\ga = (8 + 6\ga^2)^{1/3} \neq 0 \implies
r_\ga s_\ga = \frac{2(1+\ga^2)}{(8 + 6\ga^2)^{1/3}},
\]
and so $r_\ga$ and $s_\ga$ are the roots of the quadratic equation
\[
\begin{gathered}
X^2 -(8 + 6\ga^2)^{1/3} X + \frac{2(1 + \ga^2)}{(8+6\ga^2)^{1/3}} = 0.
\end{gathered}
\]
Let $\{f_1(x,y),f_2(x,y)\} = \{r_\ga x^2  + s_\ga y^2, s_\ga x^2  + r_\ga y^2\}$.
Since $(r_\ga- s_\ga)^2 = (r_\ga + s_\ga)^2 - 4r_\ga s_\ga =\ - \frac{2\ga^2}{(8+6\ga^2)^{1/3}} 
\neq 0$,
these roots are distinct, and since
$(r_{\ga} + s_{\ga})(f_3 + f_4) = 2(f_1 + f_2)$,
the equation $f_1^3 + f_2^3= f_3^3 + f_4^3$ is a Type$(\frac{r_\ga + s_\ga}2)$ family.
\end{proof}

\begin{theorem}\label{wildcase}
If 
\begin{equation}\label{E:wild}
p(x,y) = f_1^3(x,y) + f_2^3(x,y) := (a x^2 + b x y + c y^2)^3 + (d x^2 + e x y + f y^2)^3
\end{equation}
is even and a sum of two even cubes $f_3^3(x,y) + f_4^3(x,y)$,
$(b,e) \neq (0,0)$, and $(d,e,f) \neq \om^k(a,-b,c)$, then a flip of 
$f_1^3 + f_2^3 = f_3^3 + f_4^3$ is a Type$(T)$ family for some $T$ and $p$ has a 
third representation as a sum of two cubes. 
\end{theorem}
\begin{proof}
By considering the 
coefficients of $x^5y, x^3y^3, xy^5$ in \eqref{E:wild},  we need to solve
\begin{equation}\label{E:3.5}
3a^2b + 3d^2 e = 6abc + b^3 + 6def + e^3 = 3bc^2 +3ef^2 = 0.
\end{equation}
If $a=0$ in \eqref{E:3.5}, then $d^2e=0$. If $d=0$, then $a=d=0$ implies a common factor
in the quadratics, violating Lemma \ref{L:relprime}.
Hence $a=e=0$, so $b^3=0$ and $b=e=0$. These contradictions
imply that $a\neq 0$; similar arguments show that $cef \neq 0$. And now, if $b=0$, then
$d^2e = 0$ and $e \neq 0$ imply $d=0$, so $b \neq 0$ after all.  Similarly $e \neq 0$. Thus all
variables in \eqref{E:3.5} are non-zero.

By a scaling of $(x,y)$, we may assume $a=c=1$, so
\begin{equation}\label{E:wild2}
p(x,y) = (x^2 + b x y +y^2)^3 + (d x^2 + e x y + f y^2)^3
\end{equation}
is even, and \eqref{E:3.5} becomes
\begin{equation}
\begin{gathered}
3b + 3d^2e = 6b + b^3 + 6def + e^3 = 3b +3f^2e = 0.
\end{gathered}
\end{equation}
It follows that $b = -d^2e$ and $f^2 = d^2$; the remaining
equation becomes
\begin{equation}\label{E:3.8}
0 = - 6 d^2e - d^6e^3 + 6def + e^3 = e^3(1-d^6) + 6de(f-d).
\end{equation}
If $f=d$ in \eqref{E:3.8}, then $d^6=1$, so up to a power of $\om$, 
$d \in \{1,-1\}$. If $d=1$, then $e=-b, f=1$
implies that \eqref{E:wild2} is tame; if $d=-1$, then $e=-b,f=-1$ implies that $p = 0$. In
the remaining case, $f=-d$ and  $e^2(1-d^6)=12d^2$, so $e = \pm \frac{2\sqrt{3}d }{\sqrt{1-d^6}}$, 
$d^6 \neq 1$. 
By taking $y \mapsto -y$ if necessary, we may choose one square root and rewrite 
\eqref{E:wild2} as
\begin{equation}\label{E:3.9}
p(x,y) = 
 \left( x^2 - \frac{2\sqrt{3} d^3}{\sqrt{1-d^6}}\  x y + y^2\right)^3 +
\left(d x^2 +  \frac{2\sqrt{3}d}{\sqrt{1-d^6}}\ x y - dy^2\right)^3.
\end{equation}
Write \eqref{E:3.9} as $p = f_1^3 + f_2^3$. Pull $d^3$ out of the second factor and
 let $r = d^3$. A computation shows that
\[
p(x,y) = (1+r)x^6 + \frac {3(1+10r+r^2)}{1-r}\ x^4y^2 
+ \frac {3(1-10r+r^2)}{1+r}\ x^2y^4 + (1-r)y^6.
\]

We use the Sylvester algorithm (see \cite[Thm.2.1]{Re2})
to write $p$ as a sum of two cubes of even quadratics. In this way, and omitting details,
 we find that
\begin{equation}\label{E:3.10}
\begin{gathered}
p(x,y) = 
 r \left(-\frac{2+3r+ r^2}{1-r^2}\cdot  x^2 + \frac{2 - 3r + r^2}{1-r^2}\cdot y^2\right)^3  + \\
 \left(\frac{1+3r +2r^2}{1-r^2}\cdot x^2 + \frac{1 - 3r + 2r^2}{1-r^2}\cdot y^2\right)^3.
\end{gathered}
\end{equation}
Write \eqref{E:3.10} as $f_3^3 + f_4^3$, and restore $r = d^3$, so we now have
\begin{equation}\label{E:3.11}
\begin{gathered}
f_1(x,y) = x^2  - \frac{2\sqrt{3} d^3}{\sqrt{1-d^6}} x y + y^2, \quad
f_2(x,y) = d x^2 +  \frac{2\sqrt{3}d}{\sqrt{1-d^6}} x y - dy^2,\\
f_3(x,y) = -\frac{d(2+3d^3+ d^6)}{1-d^6}\cdot x^2 + \frac{d(2 - 3d^3 + d^6)}{1-d^6}\cdot y^2, \\
f_4(x,y) = \frac{1+3d^3+ 2d^6}{1-d^6}\cdot x^2 + \frac{1 - 3d^3 + 2d^6}{1-d^6}\cdot y^2.
\end{gathered}
\end{equation}
Putting this together,  \eqref{E:3.9}, \eqref{E:3.10} and \eqref{E:3.11} imply that
\[
\begin{gathered}
f_1^3(x,y) -f_4^3(x,y) = f_3^3(x,y) -f_2^3(x,y), \\
f_1(x,y) + d^2f_2(x,y) = d^2f_3(x,y)+ f_4(x,y) \quad (= (1+d^3)x^2 + (1-d^3)y^2).
\end{gathered}
\]
Thus, the wild case flips into a Type($d^2)$ family. Since $p$ is even, and 
$f_3,f_4$ are not, we also have $p = f_5^3 + f_6^3$ where $f_5(x,y) = f_3(x,-y)$
and $f_6(x,y) = f_4(x,-y)$.
\end{proof}

  \section{Equations of Type$(T)$}
 In this section we completely describe the solutions to \eqref{E:funda} of Type$(T)$. 
We begin with a probably familiar result from Diophantine analysis.
 
\begin{proposition}\label{W2} 
Suppose $\Phi(u,v) = au^2 + 2b uv + cv^2$ is a rank two quadratic form
in $\cc[x,y]$. Then any two honest solutions  $(p_i,q_i,r_i), i = 1,2$,  in binary quadratic forms
to the following equation are similar. 
\begin{equation}\label{E:dummy}
\Phi(p,q) = r^2.
\end{equation}
\end{proposition}
\begin{proof}
Write \eqref{E:dummy} as $(a_{11}p + a_{12} q)(a_{21}p + a_{22} q) = r^2$, where the
factors on the left are distinct. Since $\gcd(p,q) = 1$, 
$\gcd(a_{11}p + a_{12} q,a_{21}p + a_{22} q) = 1$ as well. It follows by unique factorization
that $ (a_{11}p + a_{12} q,a_{21}p + a_{22} q,r) = (g^2,h^2,gh)$, for suitable distinct linear forms
$g,h$. Let $[b_{ij}] = [a_{ij}]^{-1}$.  Then 
\[
(p,q,r) = (b_{11}g^2 + b_{12} h^2, b_{21}g^2 + b_{22} h^2,gh).
\] 
In particular, $(p_j,q_j,r_j)$ comes from $(g_j,h_j)$, and
the linear change $M$ taking the honest pairs of linear forms $(g_1,h_1)$ into $(g_2,h_2)$ will 
take $(p_1,q_1,r_1)$ into $(p_2,q_2,r_2)$. 
\end{proof} 
 
 \begin{lemma}\label{L:notype}
 If \eqref{E:funda} is  honest and a Type$(T)$ family, then $T(T^3-1) \neq 0$.
 \end{lemma}
 \begin{proof}
 If $T= 0$, then $f_2=-f_1$, violating honesty. Suppose $T^3 = 1$, so $T= \om^k$.
Then by $(f_3,f_4) \mapsto \om^k(f_3,f_4)$ we may assume
that $f_1 + f_2 = f_3 + f_4$.  In this case, we have
\begin{equation}\label{E:trickery}
(f_1+f_2)^2 - \frac{f_1^3 + f_2^3}{f_1+f_2} = (f_3+f_4)^2 -\frac{f_3^3 + f_4^3}{f_3+f_4} \implies
f_1f_2 = f_3f_4.
\end{equation}
This implies that $\{f_1,f_2\} = \{f_3,f_4\}$, again violating honesty. 
 \end{proof}

\begin{theorem}\label{miracle}
Suppose $\{f_1,f_2,f_3,f_4\}$ is an honest Type$(T)$ family; specifically
\begin{equation}\label{E:ayup}
\begin{gathered}
f_1^3 + f_2^3 = f_3^3 + f_4^3, \\
f_1 + f_2 = T(f_3 + f_4),\qquad T(T^3-1) \neq 0,
\end{gathered}
\end{equation}
and let $T = \la^2$.
Then there is a linear change $M$ so that 
$\{f_1\circ M,f_2 \circ M\} = \{F_{3,\la},-F_{5,\la}\}$ and 
$\{f_3\circ M,f_4 \circ M\} = \{-F_{4,\la},F_{6,\la}\}$.

\end{theorem}
\begin{proof}
As in \eqref{E:trickery}, after dividing the equations in \eqref{E:ayup} we obtain
\begin{equation}\label{E:trick2}
f_1^2 - f_1f_2 + f_2^2 = T^{-1}(f_3^2 - f_3f_4 + f_4^2).
\end{equation}
It follows that
\begin{equation}
\begin{gathered}
3f_1f_2 = (f_1+f_2)^2 - (f_1^2-f_1f_2+f_2^2) = \\
(T^2-T^{-1})f_3^2  +(2T^2+T^{-1})f_3f_4 + (T^2-T^{-1})f_4^2.
\end{gathered}
\end{equation}
But $f_1$ and $f_2$ are quadratic forms, and also the roots of the quadratic
\begin{equation}\label{E:deus}
\begin{gathered}
(X-f_1)(X-f_2) =\\
X^2 - T(f_3+f_4) X +\tfrac 13((T^2-T^{-1})(f_3^2 + f_4^2) +(2T^2+ T^{-1})f_3f_4) \\
\implies \{f_1,f_2\} = \{\tfrac{T}2 (f_3+f_4) \pm \tfrac 12 \sqrt{\Delta}\}; \\
\Delta = \tfrac 1{3T} \left( (4-T^3)f_3^2 - (4 + 2T^3)f_3f_4 + (4-T^3)f_4^2 \right) = (f_2-f_1)^2.
\end{gathered}
\end{equation}
Consider now the quadratic form $\Phi$, which has rank 2 if $T^3 \neq 1$.
\[
\Phi(u,v) =  \frac 1{3T} \left( (4-T^3)u^2 - (4 + 2T^3)u v + (4-T^3)v^2 \right).
\]
We have seen that $\Phi(f_3,f_4) = (f_2-f_1)^2$. It may be checked that
\begin{equation}
\Phi(-F_{4,\la},F_{6,\la}) = (\la^3 x^2 + 2 x y + \la^3 y^2)^2.
\end{equation}
Thus by Proposition \ref{W2}, there is a linear change $M$ so that $f_3 = -F_{4,\la}\circ M$ and 
$f_4 = F_{6,\la}\circ M$. It is routine to check that the quadratic equation 
\eqref{E:deus} then solves to give $\{f_1\circ M,f_2\circ M\}$ = $\{F_{3,\la},-F_{5,\la}\}$.
\end{proof}

\begin{proof}[Proof of Theorem \ref{T:cubic}]
Combine Theorems \ref{T:tamerep},  \ref{wildcase} and  \ref{miracle}.
\end{proof}

The historical motivation for the study of \eqref{E:funda} was to find 
parameterizations of equal sums of pairs of rational cubes, so
there is a special interest in solutions to \eqref{E:funda} in which $f_j \in \mathbb Q[x,y]$. 
Since every solution to \eqref{E:funda} is a Type$(T)$ family, we can ask a more general
question. Suppose $E \subseteq \mathbb C$ is a number field. For which values of $T = \la^2$ 
does there exist a solution to  \eqref{E:funda} with $f_j \in E[x,y]$ of Type$(T)$? 

Two partial  answers are immediate. If \eqref{E:funda} holds
with $f_j \in E[x,y]$ of Type$(\la^2)$, then \eqref{E:type}
implies that $\la^2 \in E$. On the other hand, if $\la \in E$, then \eqref{E:Naren} gives a
solution to \eqref{E:funda} with $f_j \in E[x,y]$ of Type$(\la^2)$.
What happens if $\la  \notin E$ but $\la^2 \in E$? We give negative answers in two special
cases.
\begin{theorem}\label{T:lastminute}
\ 

\noindent (i) In any solution to \eqref{E:funda} with $f_j \in \mathbb R[x,y]$ of Type$(T)$, 
we have $T > 0$.

\noindent (ii) There is no solution to  \eqref{E:funda} with $f_j \in \mathbb Q[x,y]$ of Type$(2)$.
\end{theorem}
\begin{proof}
In the first case,  \eqref{E:trick2} implies that  $f_1^2 - f_1f_2+f_2^2 = T^{-1}(f_3^2- f_3f_4+f_4^2)$.
However, for $s, t \in \mathbb R$, $s^2-st+t^2 \ge 0$, with equality only if $s=t = 0$. 
If $T^{-1} < 0$, then we must have $f_j(x,y) = 0$ for all real $x,y$. Thus $T \ge 0$, and
since $T \neq 0$, $T > 0$.

Suppose now there exists an honest
solution to  \eqref{E:funda} with $f_j \in \mathbb Q[x,y]$ of Type$(2)$,
and take multiples to ensure that $f_j \in \mathbb Z[x,y]$. By \eqref{E:deus} we have
\begin{equation}\label{E:lastminute}
 -\frac 23(f_3^2 + 5f_3f_4 + f_4^2) = (f_2-f_1)^2 \implies \\
7(f_3+f_4)^2 + 6(f_2 - f_1)^2 =  3(f_3-f_4)^2. 
 \end{equation}
We claim this is impossible. The Diophantine equation $7A^2 + 6B^2 = 3C^2$ is easily seen to
have no non-zero solutions in $\mathbb Z$. (Let $(A,B,C)$ be a solution with minimal $C$,
then $3 \ | \ A$; let $A = 3D$, so $21D^2 + 2B^2 = C^2$, hence $2B^2 \equiv C^2 \mod 3$. This
implies that $B \equiv C \equiv 0 \mod 3$, so $3 \ | \ B,C$ and $(\frac A3,\frac B3,\frac C3)$ 
is a smaller solution.)  Evaluation  of \eqref{E:lastminute} at $(x,y) \in \mathbb Z^2$ shows
that $f_3\pm f_4, f_2-f_1$ all vanish on $\mathbb Z^2$, hence are identically zero, and so the 
family is not honest after all. 
\end{proof} 

Finally, a 1595 identity of Vieta (see \cite{RR}) becomes a version of 
 \eqref{E:funda} upon clearing denominators: 
 \[
 (x(x^3-y^3))^3 +  (y(x^3-y^3))^3 = (x(x^3 +2y^3))^3 +  (-y(2x^3 + y^3))^3;
 \]
 the four quartics above are linearly independent. It seems unlikely that the methods of
 this paper are helpful when $f_j$ in \eqref{E:funda} have degree greater than two.
\section{How many ways is a sextic a sum of two cubes?}

We turn to a more general question. 
Lundqvist, Oneto, Shapiro and the author proved in \cite{LORS} 
that every binary sextic in $\cc[x,y]$ can be 
written in infinitely many different ways as a sum of three cubes of quadratic forms. 
It is natural to wonder which binary sextics can be written as a sum of two cubes, and in how
many ways. 

We need some more general notation: 
for distinct forms $F,G \in \cc[x_1\dots,x_n]$, write $X = \langle F,G\rangle$ for the 
linear subspace $\{c_1 F + c_2 G \}$, and
write $X^3 = \langle F^3, F^2G, FG^2, G^3 \rangle$; $X^3$ is
the set of all $h(F,G)$ for binary cubic forms $h$. 

\begin{theorem}\label{T:B} 
A form $p \in \cc[x_1\dots,x_n]$ of degree $3r$ can be written as $p = f_1^3 + f_2^3$ for distinct
forms $f_i$ of degree $r$ if and only if it has a factorization $p = g_1g_2g_3$ in which the $g_k$'s
are distinct but linearly dependent and $\langle f_1, f_2 \rangle =  
\langle g_1, g_2,g_3 \rangle$. If $p$ belongs to $m$ different subspaces 
$\langle F_j,G_j \rangle^3$ as above, then $N(p) \le m$. If $p$ is not divisible by the square of a form of degree $r$, then
$N(p) = m$.
\end{theorem}
\begin{proof}
In one direction,
\begin{equation}\label{E:om}
p = f_1^3 + f_2^3 \implies p =(f_1+f_2)(f_1 + \om f_2)(f_1 + \om^2 f_2):= g_1g_2g_3.
\end{equation}
If any two of the $g_i$'s are proportional
in \eqref{E:om}, then so are $f_1$ and $f_2$, and $p$ is a cube contrary
to hypothesis. For dependence, $g_j \in \langle f_1,f_2 \rangle$, 
also, $g_1 + \om g_2 + \om^2 g_3 = 0$.

Conversely if $P = g_1g_2g_3$ and  $g_1$ and $g_2$ are distinct with 
$g_3 \in  X = \langle g_1, g_2 \rangle$,  there exist $\al, \be \neq 0$ so that 
$g_3 = \al g_1 + \be g_2$. The sum of two cubes follows from an old formula
(recall that $\om - \om^2 = \sqrt{-3}$):
\begin{equation}
\begin{gathered}
p = g_1g_2g_3 = g_1g_2(\al g_1 + \be g_2) = \\
\frac 1{3\sqrt{-3}\  \al \be} \cdot \left( (\om \al g_1-  \be g_2)^3 + ( -\al g_1+ \om \be  g_2)^3 \right).\end{gathered}
\end{equation}
Suppose $p$ had two different 
expressions as a sum of two cubes of forms in $\langle f_1,f_2 \rangle$:
\[
p = (c_{1,1}f_1 + c_{2,1}f_2)^3 + (c_{3,1}f_1 + c_{4,1}f_2)^3 
= (c_{1,2}f_1 + c_{2,2}f_2)^3 + (c_{3,2}f_1 + c_{4,2}f_2)^3.
\]
Then by the linear independence of $\{f_1^{3-k}f_2^k\}$ from
Lemma \ref{L:linin}, it follows that
\[
(c_{1,1}x + c_{2,1}y)^3 + (c_{3,1}x + c_{4,1}y)^3 
= (c_{1,2}x + c_{2,2}y)^3 + (c_{3,2}x + c_{4,2}y)^3,
\]
which contradicts Lemma \ref{L:no2sums}. 

Thus, every representation of $p=f_1^3+f_2^3$ identifies the subspace
$\langle f_1,f_2\rangle^3$. Conversely, if $p \in \langle f_1,f_2\rangle^3$, then there is a cubic
form $h$ so that $p = h(f_1,f_2)$ and
\[
h(x,y) = \sum_{j=1}^2 (\al_j x + \be_j y)^3 \implies p(x,y) =  \sum_{j=1}^2 (\al_j f_1 + \be_j f_2)^3.
\]
If $p \in \langle f_1,f_2\rangle^3$, then $p$ is a sum of two cubes, unless $h$ is a cube 
(and hence so is $p$),
or $h(x,y) = (\al_1 x + \be_1 y)^2(\al_2 x + \be y)$, so  $p$ is divisible by
$(\al_1 f_1 +\be_1 f_2)^2$.
\end{proof}

Our study of sextics relies critically on the behavior of cubics as a sum of cubes.
An important corollary was known in the 19th century (see also e.g. \cite[Thm.5.2]{Re2}).
 A binary cubic $q$ is {\it square-free} if it is a 
product of three pairwise distinct linear factors. 
\begin{proposition}\label{T:0}
If $p$ is a binary cubic which is not the cube of a linear form, then $p = \ell_1^3 + \ell_2^3$ for 
distinct linear forms $\ell_j$ if and only if $p$ it is square-free, and this representation is unique,
\end{proposition}

\begin{proof}
In the general case, $f = \ell_1\ell_2\ell_3$ is a product of three distinct linear forms; any three 
such forms are linearly dependent. The other cases are $f =\ell^3$ and 
$f= \ell_1^2\ell_2$, and the necessary factorization is impossible.
\end{proof}

For Theorems \ref{T:2}, \ref{T:3}, \ref{T:46}, recall \eqref{E:biggies}.

\begin{theorem}\label{T:1}
A binary sextic $p(x,y)$ is an honest sum of two cubes  ($N(p) \ge 1$) 
if and only if one of the two 
conditions hold:  (i) $p = \ell^3q$, where $\ell$ is linear form and $q$ is a square-free cubic; or
(ii) $p$ is similar to $q(x^2,y^2)$, where $q$ is a square-free cubic, so $p$ is
similar to an even binary sextic. 
\end{theorem}

\begin{theorem}\label{T:2}
A binary sextic $p$ has $N(p) =2$ if and only if $p$ is similar to $A_t$ for $t \in \cc$, with the
 following exceptional values: $N(A_3) = 0$,  $N(A_{-1})= 1$, $N(A_0) = N(A_{15}) = 4$ and $N(A_{-5}) = 6$. 
\end{theorem}

\begin{theorem}\label{T:3}
A binary sextic $p$ has $N(p) = 3$ if and only if $p$ is similar to $B_t$ 
for $t \in \cc$, except that $N(B_{\pm 2}) = 0$,
$N(B_0)= 4$ and $N(B_{\pm 5\sqrt{-2}}) = 6$. 
\end{theorem}

\begin{theorem}\label{T:46}
The binary sextics $p$ with  $N(p) > 3$ are similar to $Q_1$ or $Q_2$:
$N(Q_1) = 4$ and $N(Q_2) = 6$; $Q_1$ is similar to $A_0,A_{15}, B_0$;
$Q_2$ is similar to $A_{-5}$ and $B_{\pm 5\sqrt{-2}}$.
\end{theorem}

\begin{proof}[Proof of Theorem \ref{T:1}]
Suppose $p = f_1^3 + f_2^3$ is a binary sextic with $N(p) \ge 1$. If $f_1$ and $f_2$ are not distinct, then $p$ is
a cube, so $f_1$ and $f_2$ are distinct. If 
$\gcd(f_1,f_2) = \ell$ is linear, then $f_1 = \ell \ell_1$ and $f_2 = \ell \ell_2$, where $\ell_1$
and $\ell_2$ are distinct. Thus, $p = \ell^3(\ell_1^3 + \ell_2^3)$ satisfies (i).
If $f_1$ and $f_2$ are relatively
prime, then by Theorem \ref{T:diag}, we may make a linear change $M$ so that
both $f_1\circ M$ and $f_2\circ M$ are even; that is, there exist distinct linear forms
$\ell_j$ so that $(f_j\circ M)(x,y) = \ell_j(x^2,y^2)$; now let $q = \ell_1^3 + \ell_2^3$; this is (ii).
\end{proof}

\begin{theorem}\label{T:nosquare}
If $p$ is a binary sextic with a square factor, then $N(p) \le 1$.
\end{theorem}
\begin{proof}
Suppose $\ell^k \ | \ p$ for a linear factor $\ell$, where $k \ge 2$.
Suppose $k \ge 3$ and $p = f_1^3 + f_2^3$ for quadratic forms $f_1,f_2$. 
Then as in Lemma \ref{L:relprime}, $\ell$ must divide at least two of $\{f_1 + \om^k f_2\}$, 
and so $\ell \ | \ f_1,f_2$,
so $p$ has no other representation as a sum of two cubes. 

Now suppose $k = 2$, and after a linear change, take $\ell = y$, so that for some
$c_j \in \mathbb C$,
\[
p(x,y) = \la y^2(x+c_1y)(x+c_2y)(x+c_3y)(x+c_4y).
\]
To apply Theorem \ref{T:0}, we need to write $p = g_1g_2g_3$ for linearly dependent
factors. If $y$ divides two of the $g_j$'s, it must divide the third, which is impossible, hence we
may assume that $g_1 = y^2$. If $N(p) \ge 2$, then after reindexing if necessary, each of 
these two different sets is dependent:
\[
\begin{gathered}
\{y^2, (x+c_1y)(x+c_2y), (x+c_3y)(x+c_4y)\},\\ \{y^2, (x+c_1y)(x+c_3y), (x+c_2y)(x+c_4y)\}.
\end{gathered}
\]
But dependence implies that  $c_1 + c_2 = c_3 + c_4$ and
$c_1 + c_3 = c_2 + c_4$, so $c_3=c_2$ and $c_4= c_1$ and 
$(x+c_1y)(x+c_2y)=(x+c_3y)(x+c_4y)$, so the factors are not distinct. 
\end{proof}

We isolate those exceptional cases in Theorems \ref{T:2} and \ref{T:3} with
square factors.
\begin{theorem}\label{T:small}
We have $N(A_3) = 0$, $N(B_{\pm 2}) = 0$, and $N(A_{-1}) = 1$.
\end{theorem}
\begin{proof}
By the first argument of the proof of Theorem \ref{T:nosquare}, since $A_3(x,y) = (x^2 + y^2)^3$,
in any representation $A_3 = f_1^3+f_2^3$, both $f_1$ and $f_2$ are multiples of $x^2+y^2$,
so that they are not distinct. This also follows from  Liouville's solution for 
Fermat's Last Theorem in polynomials (see \cite[pp.263-265]{Rib} for a proof).

We have seen that if $\ell^2$ (but not $\ell^3$) divides a sextic $p$ and $p$ has a factorization that
partitions into three dependent factors, then one of those factors must be $\ell^2$. 
Thus the only feasible  partitions for $B_{\pm 2}(x,y) = (x^3 \pm y^3)^2$ are
$\{(x\pm y)^2,(x\pm \om y)^2,(x\pm \om^2 y)^2\}$, 
which are  linearly independent; thus $N(B_{\pm 2}) = 0$.

Finally, consider $A_{-1}$, which factors as  $(x-y)^2(x+y)^2(x^2+y^2)$. Each of the two squares
must be a factor, and  $\{(x-y)^2,(x+y)^2,x^2+y^2\} \subset \langle x^2+y^2, xy \rangle$. There is a representation for $2A_{-1}$ in \eqref{E:tameo} with
$\ga = \sqrt{-4/3}$. Thus $N(A_{-1}) = 1$.
\end{proof} 

It is worth mentioning that  $A_{-1}(x,y) = (x^2 - y^2)^2(x^2 + y^2)$,
so $A_{-1}(x,y) = q_1(x^2,y^2)$, where $q_1(x,y) = (x-y)^2(x+y)$ is not square-free. 
But $\tilde A_{-1}(x,y) = A_{-1}(x+y,x-y) = 32x^4y^2+32x^2y^4 = q_2(x^2,y^2)$, where 
$q_2(x,y) = 32xy(x+y)$ {\it is} square-free. Although $A_{-1}$ and $\tilde A_{-1}$ are similar, 
$q_1$ and $q_2$ are not.

Now suppose that $N(p) \ge 2$. By Theorem \ref{T:cubic}, we know that after a linear change,
$p$ appears as the common sum in \eqref{E:threefold2}, \eqref{E:flippo1} or \eqref{E:flippo2},
and in the first two cases, $N(p) \ge 3$. Since \eqref{E:flippo1} is a linear change of 
\eqref{E:threefold2}, we may ignore it. We now apply Theorem \ref{T:B} to $p_{3,\la}$ and
to $p_{1,\la}$, which have already been conveniently split into six linear factors. There
are 15 ways to divide six factors into three unordered pairs.  

\begin{proof}[Proof of Theorems \ref{T:2}, \ref{T:3} and \ref{T:46}]
Up to a constant which can be ignored, we have
  $p_{3,\la}(x,y) = x y (x-y)(x+y) (\al x + y)(x + \al y)$, where 
$\al = \la^3 \notin \{0, -1,1\}$, which cause repeated factors. It is not hard to check the 15 
possibilities, and we suppress the details. 
In two cases, the factors  are always dependent:
\[
\begin{gathered}
\{x(\al x + y), y(x + \al y), (x+y)(x-y)\} = \langle \al x^2 + x y, xy + \al y^2 \rangle,\\
 \{x(x + \al y), y(\al x + y), (x+y)(x-y)\} = \langle  x^2 + \al x y, \al xy + y^2 \rangle.
 \end{gathered}
 \]
There are two cases when
there are multiple dependencies. If $\al \in \{\pm2,\pm\frac 12\}$,  there are
two additional cases of dependency, and if $\al = \pm i$, there are four additional cases.
Thus, $N(p_{3,\la}) = 2$ for $\la(1-\la^6) \neq 0$ unless $\al \in \{\pm 2, \pm \frac 12, \pm i\}$.

If $\al = \la^3 = \pm i$, then up to powers of $\om$, $\la^2 = -1$. In the language of 
Theorem \ref{T:tamerep}, $\frac{r_\gamma + s_\gamma}2 = \la^2 \implies
r_\gamma + s_\gamma = -2 = (8 + 6\gamma^2)^{1/3} 
\implies 3(1+\gamma^2) = - 5$, so $p_{3,\pm i}$ is similar to $A_5$.
If $\al = \pm2, \pm \frac 12$, then $\la^2 = 2^{\pm 2/3}$ and 
$r_\gamma + s_\gamma = 2^{1/3}, 2^{5/3}  \implies 8 + 6\gamma^2 = 2, 32
\implies 3(1+\gamma^2) = 0, 15$, so $p_{3,\la}$  is similar to $A_0 = Q_1$ or 
$A_{15}$. 

Up to a constant, 
\[
p_{1,\la}(x,y) =(\la x + y) (\la x + \om y)(\la x + \om^2y)
(x + \la y)(x + \la \om y)(x + \la \om^2 y).
\]
 As we would hope, there are three cases in which the factors are always dependent:
\begin{equation}\label{E:p1factors}
\begin{gathered}
\{(\la x + y)(x + \la y), (\la x + \om y)(x + \la \om^2 y), (\la x + \om^2y)(x + \la \om y)\},\\
\{(\la x + y)(x + \la \om y), (\la x + \om y)(x + \la y), (\la x + \om^2y)(x + \la \om^2 y)\}, \\
\{(\la x + y)(x + \la \om^2 y), (\la x + \om y)(x + \la \om y), (\la x + \om^2y)(x + \la y)\}; 
 \end{gathered}
\end{equation}
the subspaces are  $\langle x^2 + \om^k y^2, xy \rangle$.
There are a few cases with multiple dependencies: when $\la = \pm i \implies \al = \pm i$,
there is one extra case. In this case, $p_{1,\pm i}(x,y) = \pm 2i Q_1(x,y)$. The other
cases in which a dependency occurs are when $\la^4 + 4\la^2 + 1$ = 0, up to $\la \mapsto 
\om^k \la$. For example, suppose
\[
\begin{gathered}
 \{(\la x + y)(x + \la y),( x +\la \om y)( x +\la \om^2 y), (\la x + \om y)(\la x + \om^2 y) \} \\
 = \{\la x^2 + (\la^2 + 1) xy + \la y^2, x^2 - \la x y + \la y^2, \la^2 x^2 - \la x y + y^2\}
 \end{gathered}
 \]
 is linearly dependent. This happens if and only if 
 \[
 \begin{vmatrix}
 \la & \la^2+1 & \la \\ 1 & -\la & \la^2 \\ \la^2 & - \la & 1
 \end{vmatrix}
= (\la^2-1)(\la^4 + 4\la^2 + 1) = 0.
 \]
In computations that Ramanujan could probably do in his sleep,
\begin{equation}
\begin{gathered}
\la^4 + 4\la^2 + 1 = 0 \implies \la^2 = -2 \pm \sqrt 3 \implies 
\la = \pm \left( \tfrac{\sqrt 6 \pm' \sqrt 2}2\right) i \\ \implies \la^3 + \la^{-3} = \pm 5\sqrt{-2}.
\end{gathered}
\end{equation}
Since $B_t$ and $B_{-t}$ are similar via $y \mapsto -y$,  we focus on $B_{5\sqrt{-2}}$.
Let $\eta = \frac{\sqrt 6 + \sqrt 2}2$, so $\eta i$ is a root. 
We have a linear change with bizarre coefficients:
\begin{equation}\label{E:bizarre}
\begin{gathered}
B_{5\sqrt{-2}}(\ze_8^2\eta x + \ze_8y, x + \ze_8^3\eta y) 
= 54\ze_8^3\eta^3 Q_2(x,y),
\end{gathered}
\end{equation}
showing that $B_{5\sqrt{-2}}$ is similar to $Q_2$. We give a geometric explanation for
\eqref{E:bizarre} in the next section. 
\end{proof}

The instance of \eqref{E:funda} with the simplest coefficients is probably
\begin{equation}\label{E:cubic}
\begin{gathered}
(x^2 + x y - y^2)^3 + (x^2 - x y - y^2)^3  = 2(x^2)^3 + 2(-y^2)^3 = 2 x^6 - 2y^6 \\
= (\om x^2 + x y - \om^2 y^2)^3 + (\om x^2 - x y - \om^2 y^2)^3 \\= 
(\om^2 x^2 + x y - \om y^2)^3 + (\om^2 x^2 - x y - \om y^2)^3.
\end{gathered}
\end{equation}
With $(x,y) \mapsto (x+y,x-y)$, \eqref{E:cubic} is due to  Girardin in 1910 (see \cite[p.550]{D}; 
 the earliest exact version of \eqref{E:cubic} I've found is by Elkies in 1995 (see \cite[p.542]{DG}). 
Observe that \eqref{E:cubic} is simply \eqref{E:threefold} with $\la = i$ and $y \mapsto iy$, and
it also a scaling of $Q_1$. (We have $2x^6 -2y^6 = Q_1(r x, s y)$ if $r^6 = 2, s^6 = -2$.)
Unsurprisingly, since $\la = i$, a flip of \eqref{E:cubic} is similar to $Q_2$:
\begin{equation}\label{E:Q1flip}
(\om x^2 + x y - \om^2 y^2)^3 - (\om^2 x^2 + x y - \om y^2)^3 = -3\sqrt{-3}(x^5y-xy^5).
\end{equation}
Finally, we remark that while \eqref{E:cubic} is presented as a Type$(-1)$ family, we have
\[
(x^2 + x y - y^2) + (x^2 - x y - y^2) = 2^{-1/3}(2^{1/3}x^2 + (-2^{1/3}y^2)),
\]
which gives a Type$(2^{2/3})$ family from \eqref{E:tameo}, with $y\mapsto iy$.
 Thus the Type parameter may vary when more than three
representations occur. 

\section{More on the extra representations}
As we saw in the last section, 
there are two special cases of sextics with more than three representations and we
treat them separately. First, note that
\[
Q_1(x,y) = x^6 + y^6 = A_0(x,y) = B_0(x,y); A_{15}(x,y) = \tfrac 12 A_0(x+y,x-y).
\]
For purposes of analyzing the
factorizations, we note that with $\la = i$, it is easier to use powers of $\nu:= \ze_{12}$:
\[
Q_1(x,y) = (x - \nu y)(x-\nu^3 y)(x-\nu^5 y)(x - \nu^7 y)(x-\mu^9 y)(x - \nu^{11}y).
\]
Keeping in mind that $i = \nu^3, \om = \nu^4$, and rearranging
\eqref{E:p1factors} a bit, we have that
the three dependent factorizations of $Q_1$ are:
\[
\begin{gathered}
\{(x+ \nu y)(x + \nu^{11} y), (x + \nu^3 y)(x + \nu^9 y), (x + \nu^5 y)(x + \nu^7 y)\},\\
\{(x+ \nu y)(x + \nu^{3} y), (x + \nu^7 y)(x + \nu^9 y), (x + \nu^5 y)(x + \nu^{11} y)\},\\
\{(x+ \nu y)(x + \nu^{7} y), (x + \nu^3 y)(x + \nu^5 y), (x + \nu^9 y)(x + \nu^{11} y)\}.
 \end{gathered}
\]
These live in $\langle x^2 + y^2, x y \rangle , \langle x^2 + y^2, \om x y \rangle , 
\langle x^2 + \om^2 y^2, x y \rangle $ respectively. The fourth dependent factorization is
\[
\{(x+ \nu y)(x + \nu^{7} y), (x + \nu^3 y)(x + \nu^9 y), (x + \nu^5 y)(x + \nu^{11} y)\} 
\subseteq \langle x^2, y^2 \rangle.
\]
The best way of visualizing the four equal pairs of sums seems to be \eqref{E:cubic}.

The other case is somewhat more mysterious. Since $Q_2(x,y) = xy(x^4-y^4)$, it is simple to 
work out all fifteen factorizations into three quadratics. The following six are dependent:
\[
\begin{gathered}
\{xy,(x+y)(x+iy),(x-y)(x-iy)\} \subseteq \langle x^2+i y^2, xy \rangle,\\
\{xy,(x+y)(x-iy),(x-y)(x+iy)\} \subseteq \langle x^2-i y^2, xy \rangle, \\
\{x(x+y),y(x-y),(x+iy)(x-iy)\} \subseteq \langle x^2+xy, x^2+y^2 \rangle, \\
\{x(x+iy),y(x-iy),(x+y)(x-y)\} \subseteq \langle x^2+ixy, x^2-y^2 \rangle, \\
\{x(x-y),y(x+y),(x+iy)(x-iy)\} \subseteq \langle x^2-xy, x^2+y^2 \rangle, \\
\{x(x-iy),y(x+iy),(x+y)(x-y)\} \subseteq \langle x^2-ixy, x^2-y^2 \rangle.
\end{gathered}
\]
We could simply  write $Q_2$ explicitly
as an element in $\langle F, G \rangle^3$ in these six cases.
 It is more interesting to derive them from earlier work; see \eqref{E:Q21}, 
 \eqref{E:Q22}, \eqref{E:Q23} below.

First, observe that $r_{3,i}(x,y) = r_{3,-i}(x,y) = -3\sqrt {-3}\ Q_2(x,y) =(\sqrt{-3})^3 Q_2(x,y)$.
One would think that this gives four representations of $Q_2$, coming from \eqref{E:flippo2};
however the representation for $\la = -i$ is a permutation of that from $\la = i$, and there are
only two distinct ones: 
\begin{equation}\label{E:Q21}
\begin{gathered}
 -3\sqrt {-3}\ Q_2(x,y) = (\nu^5 x^2 + xy + \nu y^2)^3 + (\nu^7 x^2 - xy +\nu^{11}y^2)^3, \\
  3\sqrt {-3}\ Q_2(x,y) = (\nu^{11} x^2 + xy + \nu^7 y^2)^3 + (\nu x^2 - xy +\nu^{5}y^2)^3. \\
\end{gathered}
\end{equation}
These come from $\langle x^2-ixy, x^2-y^2 \rangle$ and $\langle x^2+ixy, x^2-y^2 \rangle$
respectively. However, $Q_2(x,y) = -iQ_2(x,iy)$, so 
\begin{equation}\label{E:lasttrick}
Q_2(x,y) = f_1(x,y)^3 + f_2(x,y)^3 \implies Q_2(x,y) = (if_1(x,iy))^3 + (if_2(x,iy))^3.
\end{equation}
 In this way,  we immediately obtain two more representations:
\begin{equation}\label{E:Q22}
\begin{gathered}
 -3\sqrt {-3}\ Q_2(x,y) = (\nu^{10} x^2 + xy + \nu^8 y^2)^3 + (\nu^8 x^2 - xy +\nu^{10}y^2)^3, \\
  -3\sqrt {-3}\ Q_2(x,y) = (\nu^{4} x^2 + xy +\nu^{2}y^2)^3+ (\nu^2 x^2 - xy + \nu^4 y^2)^3. \\
\end{gathered}
\end{equation}
These are in $\langle x^2+xy, x^2+y^2 \rangle$ and $\langle x^2-xy, x^2+y^2 \rangle$, as one
would expect; \eqref{E:lasttrick} simply permutes the equations, and we get no more. 
Since $\nu^4 = \om$ and $\nu^2 = - \om^2$, the second equation in \eqref{E:Q22} recovers
\eqref{E:Q1flip}.

Finally, $Q_2(\frac{x+y}{\sqrt 2}, \frac{x-y}{\sqrt 2}) = Q_2(x,y)$, so after some simplification, 
we obtain the final two representations of $Q_2$: 
\begin{equation}\label{E:Q23}
\begin{gathered}
 6\sqrt {-6}\ Q_2(x,y) = (\ze_8^5 x^2 + \sqrt{6}\ xy +\ze_8^7y^2)^3+ (\ze_8 x^2 + \sqrt{6}\ xy + \ze_8^3 y^2)^3.\\
 6\sqrt {-6}\ Q_2(x,y) = (\ze_8^7 x^2 - \sqrt{-6}\ xy +\ze_8^5y^2)^3+ (\ze_8^3 x^2 - \sqrt{-6}\ xy + \ze_8 y^2)^3.
\end{gathered}
\end{equation}
These are in  $\langle x^2+iy^2, xy \rangle$ and $\langle x^2-iy^2, xy \rangle$. Although it 
might seem  daunting to consider checking whether any two of these six equations are similar, the 
fact that they live in different subspaces shows that this is impossible.

Finally, we discuss the connection  of $Q_2$ and $B_{5\sqrt{-2}}$. To do so, we need
an old idea of Felix Klein; see also \cite[p.731]{Re3}. 
Associate to each non-zero linear form $\ell(x,y) = s x - t y$
the image of $t/s \in \cc^*$ on the unit sphere $S^2$ under the
Riemann map and vice-versa. (Assign $\ell(x,y) = y$ to $\infty$ and $(0,0,1)$.)
The {\it Klein set} of $p(x,y)= \prod_{j=1}^k (s_j x - t_j y)$
  is the image of the $k$ points $t_j/s_j$ on $S^2$ under the Riemann map. 
Every rotational symmetry of the Klein set of $p$ has an interpretation as a 
symmetry of $p$ under a linear change. 

There are two particularly symmetric six-point sets on $S^2$.
One is a hexagon along a great circle, say the equator. Note that 
$Q_1(x,y)= x^6 + y^6 = \prod_{j=0}^5(x + \zeta_{12}^{2j+1} y)$
has such a hexagon as its Klein set. The other natural  choice is the vertex set of a regular 
octahedron, and the Klein set of $Q_2$ is $\{\pm e_k\}$:
\[
 Q_2(x,y) = xy(x-y)(x+y)(x-iy)(x+iy) = xy(x^4-y^4).
\]
The two symmetries of $Q_2$ mentioned above come from rotating the octahedron by
$\frac{\pi}2$ on the $z$-axis andon the $y$-axis.

One may rotate an octahedron so that the top and bottom are antipodal triangular
faces parallel to the equator.  One set of coordinates of the vertices is:
 \begin{equation}\label{E:kappa}
 \left\{ \pm \left(\tfrac 2{\sqrt 6},0,\tfrac{\sqrt{2}}{\sqrt{6}}\right),
  \pm \left(\tfrac {-1}{\sqrt 6},\pm' \tfrac{\sqrt{3}}{\sqrt{6}}, \tfrac{\sqrt{2}}{\sqrt{6}}\right)\right\}.
\end{equation}
The cubic which corresponds to the triangle in the northern hemisphere is
\[
(x - \la_0 y)(x - \om \la_0 y)(x - \om^2\la_0 y) = x^3 - \tfrac{5 + 3\sqrt 3}{\sqrt 2} y^3,\quad  \la_0 =
 \tfrac{\sqrt 6 + \sqrt 2}2.
\]
Similarly, the cubic for the southern hemisphere is
\[
(x +\la_0^{-1} y)(x + \om \la_0^{-1} y)(x + \om^2\la_0^{-1} y) = x^3 + \tfrac{5 + 3\sqrt 3}{\sqrt 2} y^3.
\]
Multiplying these together, we get another Klein polynomial for the octahedron:
\[
\tilde Q_2(x,y) = x^6 - 5\sqrt 2\ x^3 y^3 - y^6 \implies \tilde Q_2(x,iy) = x^6 + 5\sqrt 2 i x^3 y^3 + y^6 = B_{5\sqrt{-2}}(x,y).
\]
The rotation relating $\{\pm e_k\}$ into \eqref{E:kappa} inspired the coefficients of
\eqref{E:bizarre}. 

There are, in general, $\frac {(3r)!}{3!(r!)^3}$ ways to arrange the $3r$ linear factors 
of a form $p$ into three factors of degree $r$, and by Theorem \ref{T:B}, this gives an upper
bound on the number of ways to write $p$ as a sum of two cubes. It would be interesting
to know how the actual bound grows for $p$. The natural analogues of $Q_1,Q_2$ are
$x^{3r} + y^{3r}$ and $xy(x^{3r-2} - y^{3r-2})$. 

\section{Other approaches to sums of two cubes} 

The proof of the Euler-Binet parameterization of all solutions, found for example in 
 \cite[pp.199-201]{HW}, can easily be adapted
to fields of characteristic zero. For our purposes, we look at rational functions over $\mathbb C$.
\begin{theorem}[Euler-Binet]\label{T:EB}
Suppose $F= \cc(x_1,\dots,x_n)$ and suppose
\begin{equation}\label{E:eb1}
p = f_1^3 + f_2^3 = f_3^3 + f_4^3.
\end{equation}
for pairwise distinct $f_1,f_2,f_3,f_4 \in F$. 
Then there exist  $\mu, a, b \in F$ so that
\begin{equation}\label{E:ebrats}
\begin{gathered}
f_1 = \mu(1-(a-3b)(a^2+3b^2)), \quad f_2 = \mu((a+3b)(a^2+3b^2)-1),\\
f_3 = \mu((a + 3b)-(a^2 + 3b^2)^2), \quad f_4 = \mu ((a^2 + 3b^2)^2 - (a - 3b)).
\end{gathered}
\end{equation}
Conversely, if $f_1, f_2, f_3,f_4$ are given by \eqref{E:ebrats} in terms of $\mu, a, b$, then
\begin{equation}
p = f_1^3 + f_2^3 =f_3^3 + f_4^3 = 18\mu^3b(a^2 + 3b^2)(1 - (a + b)^3 - (a - b)^3 + (a^2 + 3b^2)^3)). 
\end{equation}
\end{theorem}
\begin{proof}
Define $g_i$'s by
\begin{equation}\label{E:some}
f_1 = g_1 + g_2,\quad f_2 = g_1 - g_2,\quad f_3 = g_3 + g_4,\quad f_4 = g_3 - g_4, 
\end{equation}
so that \eqref{E:eb1} becomes 
\[
p = 2g_1(g_1^2 + 3g_2^2) = 2g_3(g_3^2 + 3g_4^2).
\]
Since $p \neq 0$, $g_1^2 + 3g_2^2\neq 0$ as well, and we may define
\begin{equation}\label{E:eb25}
a = \frac{g_1g_3+3g_2g_4}{g_1^2+3g_2^2}, \quad b = \frac{g_1g_4-g_3g_2}{g_1^2 + 3g_2^2}.
\end{equation}
Observe that
\begin{equation}\label{E:eb3}
\begin{gathered}
ag_1 -3bg_2 = g_3;\qquad  bg_1 + ag_2 = g_4; \qquad
a^2 + 3b^2 = \frac{g_3^2 + 3 g_4^2}{g_1^2+3g_2^2} = \frac{g_1}{g_3}.
\end{gathered}
\end{equation}
(In the original derivation, taken over $\mathbb Q$, $(a,b)$ are defined
by $a \pm b \sqrt{-3}= \frac{g_3 \pm g_4\sqrt{-3}}{g_1\pm g_2\sqrt{-3}}$,
which is unambiguous. We cannot do this here, because some coefficient of $g_j$ might involve
$\sqrt{-3}$,  but \eqref{E:eb25} recaptures the essence.) Now let
\begin{equation}\label{E:eb4}
\begin{gathered}
c = a(a^2+3b^2)-1, \qquad d = 3b(a^2+3b^2) \\ \implies  cg_1 - dg_2 
= (a^2 + 3b^2)(ag_1-3bg_2) - g_1 = (a^2+3b^2)g_3 - g_1=0,
\end{gathered}
\end{equation}
so $cg_1 = dg_2$.
Suppose $c=d=0$. Looking at $d=0$, $a^2+3b^2=0$  implies
$c = -1$), so $b=0$, and $ag_1 = g_3$,  and
$ag_2 = g_4$ by \eqref{E:eb3}, so that $af_1 = f_3$ and $af_2 = f_4$ implying
that \eqref{E:eb1} is not honest. Thus $(c,d) \neq (0,0)$, and
we write  $(g_1,g_2)$ with $\mu \in F$ as
\begin{equation}\label{E:g1g2}
g_1 = \mu d = 3 \mu b(a^2+3b^2),\quad  g_2 = \mu c = \mu (a(a^2+3b^2)-1).
\end{equation}
Now solve for $g_3$ and $g_4$ from \eqref{E:eb3}:
\begin{equation}\label{E:g3g4}
\begin{gathered}
g_3 = a g_1 -3b g_2 = 3\mu b, \quad g_4 = bg_1 + ag_2 = \mu((a^2 + 3b^2)^2 - a).
\end{gathered}
\end{equation}
Plug back in to \eqref{E:some} and \eqref{E:eb25}  to get \eqref{E:ebrats}.
\end{proof}

 \begin{corollary}\label{C:EB}
Suppose $f_1,f_2,f_3, f_4 $ are forms of degree $k$ satisfying \eqref{E:eb1}.
Then up to a possible common factor, there exist forms $p, q, r$ of degree $\le 2k$ so that
\begin{equation}\label{E:ebratpar}
\begin{gathered}
f_1 = r(r^3-(p-3q)(p^2+3q^2)), \quad f_2 = r((p+3q)(p^2+3q^2)-r^3),\\
f_3= r^3(p + 3q)-(p^2 + 3q^2)^2, \quad f_4 = (p^2 + 3q^2)^2 - r^3(p - 3q).
\end{gathered}
\end{equation}
\end{corollary}
\begin{proof}
Define $f_1,f_2,f_3,f_4$ as above, and define $a$ and $b$ via \eqref{E:eb25} 
as rational functions with a common denominator, subject to possible cancellation:
\begin{equation}\label{E:7.n}
a = \frac{p(x,y)}{r(x,y)}, \qquad b = \frac{q(x,y)}{r(x,y)}.
\end{equation}
The expressions for $f_3,f_4$ have a formal denominator of $r^4$, so we take 
$\mu(x,y) = r^4(x,y)$, with the 
understanding that cancellation may occur. By substituting \eqref{E:7.n} into
\eqref{E:ebrats}, we obtain \eqref{E:ebratpar}.
\end{proof}
Applying this to the quadruple $(f_1,f_2,f_3,f_4) = 
(F_{6,\la},-F_{4,\la}, F_{3,\la},-F_{5,\la})$,  there is much 
cancellation and
\begin{equation}
\begin{gathered}
a = -\frac{x^2 + y^2}{2\la\ xy}, \quad b  = 
\frac{-i(x^2-y^2)}{2\sqrt{3}\la\ xy}, \quad \mu =  r^4 x y,
\end{gathered}
\end{equation}
so that $p$ and $q$ are quadratic, and $r$ is linear. Other choices for the $f_j$'s lead
to $p,q,r$ of higher degree. There are $3^4\cdot 4! = 1944$ ways to arrange the $f_i$'s, 
counting cube roots of
unity, and we cannot assert that a simpler set of parameters doesn't exist. In the 
famous Ramanujan case of $12^3 + 1^3 = 10^3 + 9^3 = 1729$, the integral version of 
\eqref{E:ebrats} comes from $(a,b,\mu) = (\frac{10}{19}, \frac{7}{19},-\frac{361}{42})$, but
permuting 9 and 10 means that we need denominators of 266 and 333. On the other hand,
 the same identity
flipped as $10^3 + (-1)^3 = (-9)^3 + 12^3$ comes from $(a,b,\mu) = (- \frac 32, \frac12,1)$.

The other standard approach to equal sums of cubes arises from point-addition on the
curve $X^3 + Y^3 = A$; see e.g. \cite{Sil}.
Assuming that $(X,Y) = (X_1,Y_1), (X_2,Y_2)$ lie on this curve, 
the cubic equation $(tX_1 + (1-t)Y_1)^3 + (tX_2 + (1-t)Y_2)^3=A$ has two solutions
$t=0,1$, and so the third may be computed; after simplification,
\begin{equation}\label{E:pointsum}
\begin{gathered}
X_3 = \frac{A(X_1-X_2) +
  Y_1Y_2(X_2Y_1-X_1Y_2)}{(X_1^2X_2 + 
  Y_1^2Y_2)- (X_1X_2^2 + Y_1Y_2^2)}, \\
Y_3 = \frac{A(Y_1-Y_2) +
  X_1X_2(X_1Y_2-X_2Y_1)}{(X_1^2X_2 + 
  Y_1^2Y_2)- (X_1X_2^2 + Y_1Y_2^2)}.
  \end{gathered}
\end{equation}

This computation (usually done over $\mathbb Q$),  is still valid when
$X_i,Y_i$ are polynomials. Of course, the denominator means that the new solution is
usually composed of rational functions. Somewhat astonishingly, \eqref{E:pointsum} is applicable
to \eqref{E:threefold2}, and we present a theorem whose only proof is direct computation. 

\begin{theorem}\label{T:add}
If we take $(X_1,Y_1) = (F_{1,\la}, F_{2,\la})$, 
$(X_2,Y_2) = (F_{3,\la},F_{4,\al})$ and $A = p_{1,\la}(x,y)$ in \eqref{E:pointsum}, then
  $(X_3,Y_3) = (F_{5,\la},F_{6,\la})$.
\end{theorem}

More generally, if we take the parameterizations from \eqref{E:ebrats} to
add $(f_1,f_2)$ and $(f_3,f_4)$, we obtain denominators. But 
if we add $(f_1,-f_4)$ and $(f_3,-f_2)$, which come from the flip $f_1^3-f_4^3=(-f_2)^3+f_3^3$,
we obtain a third polynomial solution which is apparently new . 
\begin{equation}\label{E:curveadd}
\begin{gathered}
f_1^3-f_4^3=-f_2^3+f_3^3 = (\mu(1 + 2a(a^2 + 3b^2)))^3  -(\mu(2a + (a^2+3b^2)^2))^3. 
\end{gathered}
\end{equation}
A few caveats: even though \eqref{E:ebrats} is a complete parameterization of solutions to 
two equal sums of two cubes; \eqref{E:curveadd} is {\it not} a complete parameterization of 
solutions to three equal sums of two cubes. An extremely tedious application of Theorem
\ref {T:B}
to the three flips of \eqref{E:ebrats} shows that this is the only bonus representation.

As is the case with $\mathbb Q$, there can
be arbitrarily large sets of equal pairs of sums of two cubes. For example, Rouse and the
author give in \cite{RR} 
the complete (infinite) solution to the solution over rational functions of:
\[
x^3 + y^3 = \left(\frac{p(x,y)}{r(x,y)}\right)^3 + \left(\frac{q(x,y)}{r(x,y)}\right)^3, \qquad
p, q, r \in \cc[x,y].
\]
for rational functions $(p/r,q/r)$. Clearing the denominator in any finite family of sums
$x^3 + y^3 = (\frac{p_i}{r_i})^3 +  (\frac{q_i}{r_i})^3, 1 \le i \le N$, gives a set of $N$ equal sums.

We may also take an invariant-theory approach to $N(p) \ge 1$.
In any sum of two cubes of quadratic forms:
\[
\sum_{j=1}^2(\al_{j0}x^2 + \al_{j1}xy + \al_{j2} y^2)^3 = \sum_{k=0}^6c_k x^{6-k}y^k, 
\]
the seven $c_k$'s are cubic polynomials in the six $\al_{j\ell}'s$,
and since $7>6$, we know that the $c_k$'s must be algebraically
dependent. There are $\binom {n+6}6$ monomials in the
$c_j$'s of degree $n$; these are forms of degree $3n$ in the
$\al_{j\ell}'s$, which 
comprise a vector space of dimension $\binom {3n+5}5$.  Eventually, 
$\binom{n+6}6 > \binom {3n+5}5$, so there must be dependence at some degree $n$. 
Unfortunately, the smallest $n$ for which this happens is $n=1442$.

We can be less brute-force and  apply Theorem \ref{T:B}. 
Suppose our given cubic $p$ is a sum of two cubes, factor it
and expand it in the usual way. Write $p$ as
\[ 
\sum_{k=0}^6 c_k x^{6-k}y^k = c_0\left(x^6 + \sum_{k=1}^6 e_k
  x^{6-k}y^k\right) = c_0\prod_{j=1}^6(x + r_j y),
\]
where the $e_k$'s are the elementary symmetric functions in the $r_j$'s.  
As noted earlier, there are 15 ways to divide the 6 $r_j$'s into 3 pairs of roots, and the condition
that the quadratic factors be dependent is equivalent to the vanishing of
\[
H(r):= \prod_{\ell=1}^{15} \begin{vmatrix} 1 & 1 & 1 \\
r_{\sigma_{\ell}(1)} +  r_{\sigma_{\ell}(2)} & r_{\sigma_{\ell}(3)} +
r_{\sigma_{\ell}(4)} & r_{\sigma_{\ell}(5)} +  r_{\sigma_{\ell}(6)} \\
r_{\sigma_{\ell}(1)}r_{\sigma_{\ell}(2)} &
r_{\sigma_{\ell}(3)}r_{\sigma_{\ell}(4)} &
r_{\sigma_{\ell}(5)}r_{\sigma_{\ell}(6)} \end{vmatrix}.
\]
where the product is taken over a suitable subset of $S_6$. (Of course $H(r) = 0$
even if the factors are dependent, so this is a necessary but not sufficient condition.)
Mathematica can compute $H(r)$ without too much
difficulty, and in a few hours transform it into a symmetric function 
in the $e_k$'s of degree 15. Now write  $e_k = c_k/c_0$,  make
the substitution and multiply by $c_0^{15}$ to get the relation. It
has 1360 terms and is {\it isobaric} in the old
sense: each monomial $\prod c_k^{m_k}$ in the product has $\sum m_k = 15, \sum km_k = 45$.
It seems likely that this is the skew invariant called $I_{15}$ in the old literature.  For more 
information, see \cite{E}, especially \S143, \S244 and Examples 20 and 21 on pp.315-6. The
original discovery is attributed there to Joubert.

Finally, here are some of the quadratic parameterizations of \eqref{E:funda} which
can be found in the literature.
The earliest one found in \cite[p.554]{D} was in  J. R. Young's 1816 book {\it Algebra}, in
S. Ward's edition of 1832, and in 1895, by the self-taught mathematician Artemas Martin 
(see \cite{AC}) in a journal he wrote, edited and typeset:
\begin{equation}\label{E:young}
\begin{gathered}
(x^2 + 16xy - 21y^2)^3 + (-x^2 +16xy + 21y^2)^3 +  (2x^2 -4xy + 42y^2)^3  \\ = (2x^2 + 4xy +
42y^2)^3.
\end{gathered}
\end{equation}
This is a Type(4) family.
In fact, Young presented a one-parameter family of such solutions, of Type($n^2$), which 
homogenizes to 
\begin{equation}
\begin{gathered}
(nx^2 -6nxy +3(n^7-n)y^2)^3 + (-x^2 + 6n^3 x y + 3(n^6-1)y^2)^3 \\ 
= (n x^2 + 6n x y + 3(n^7-n) y^2)^3 + (-x^2 - 6n^3 x y + 3(n^6-1)y^2)^3.
\end{gathered}
\end{equation}
By Theorem \ref{T:cubic}, these are similar to the Narayanan  solutions from a century later,
and since their sum is an even polynomial, there isn't a third representation.

S\'andor \cite{San} gave a beautiful solution to \eqref{E:funda} as a conditional
polynomial identity.
 (In 1873, Korneck \cite[p.556]{D} (see \cite[p.122]{San}) gave a similar  family of identities.)
He showed that if $(w_1,w_2,w_3,w_4) \in \mathbb C^4$ satisfy $w_1^3 +
w_2^3 = w_3^3 + w_4^3$, then a quadratic solution to $a^3+ b^3=c^3+d^3$ is given by
the Type$(\frac{w_4-w_2}{w_1-w_3})$ family.
\begin{equation}
\begin{gathered}
a =  w_2(w_1-w_3)x^2 + (w_1^2-w_3^2) x y + w_4(w_4-w_2)y^2, \\
b =  -w_3(w_1-w_3)x^2 + (w_2^2-w_4^2) x y - w_1(w_4-w_2)y^2, \\
c = w_4(w_1-w_3)x^2 + (w_1^2-w_3^2) x y + w_2(w_4-w_2)y^2, \\
d = -w_1(w_1-w_3)x^2 + (w_2^2-w_4^2) x y - w_3(w_4-w_2)y^2.
\end{gathered}
\end{equation}

Hirschhorn has written several papers which explore Ramanujan's approach
to \eqref{E:funda} and related questions. In \cite{H1}, he conjectured that an ``amazing" identity
of Ramanujan in his ``Lost Notebook" could be proved via the Type$(4)$ identity
\begin{equation}
(x^2 + 7xy - 9y^2)^3 + (2x^2 - 4xy + 12y^2)^3 = (2x^2+10y^2)^3 + (x^2
- 9xy - y^2)^3,
\end{equation}
and in \cite[p.388]{H2}, he derived this as a special case of a more general
formula, which homogenizes to the Type$(n^2)$ identity:
\begin{equation}
\begin{gathered}
(3x^2 + 6n^3xy +(1-n^6)y^2)^3 + (3nx^2 - 6nxy +(n^7-n)y^2 )^3 \\
= (3x^2 - 6n^3xy +(1-n^6)y^2)^3 + (3nx^2 + 6nxy +(n^7-n)y^2 )^3.
\end{gathered}
\end{equation}


\end{document}